\documentclass[12pt]{amsart}
\usepackage{latexsym,color,amsmath,amsthm,amssymb,amscd,amsfonts}

\usepackage{hyperref}

\newtheorem{thm}{Theorem}[section]
\newtheorem{lem}[thm]{Lemma}

\newtheorem{rmk}[thm]{Remark}

\begin{document}


\title{On inscribed equilateral simplices in normed spaces}

\author[B. Gonz\'alez Merino]{Bernardo Gonz\'alez Merino}\email{bgonzalez4@us.es}

\thanks{
This research is a result of the activity developed within the framework of the Programme in Support of Excellence Groups of the Regi\'on de Murcia, Spain, by Fundaci\'on S\'eneca, Science and Technology Agency of the Regi\'on de Murcia. Partially supported by Fundaci\'on S\'eneca project 19901/GERM/15, Spain, and by MINECO Project MTM2015-63699-P Spain.}

\date{\today}\maketitle

\begin{abstract}
In this paper we prove in certain n-dimensional normed spaces $X$ 
the existence of full-dimensional equilateral simplices
of large size inscribed to the unit ball $B$.
This extends the construction of Makeev \cite{Mak} in dimension 4 and we also compute
an example of a space in which the idea cannot be applied.
\end{abstract}

For any n-dimensional normed space $X$, its \emph{equilateral dimension} $e(X)$ is the maximum number of pairwise equidistant points. 
The study of this quantity goes back to the '50s, see Petty \cite{Pet} and the references therein.
Danzer and Gr\"unbaum \cite{DaGr} proved that $e(X)\leq 2^n$, with equality if and only if $X$ is isometric to $\ell^n_\infty$. For the lower bound,
it is believed that $e(X)\geq n+1$ (cf.~\cite{BMP,Pet}). In the planar case, this is a simple exercise. The cases of dimension 3 and 4 were proven
by Petty \cite{Pet} and Makeev \cite{Mak}, respectively. In higher dimensions, there are several results about $e(X)$.
For instance, if the Banach-Mazur distance from $X$ to $\ell^n_2$ is at most $1+1/n$ then Brass \cite{Br} and Dekster \cite{De}
proved that $e(X)\geq n+1$, and if its distance to $\ell^n_\infty$ is at most $2$, then Averkov \cite{Av} (see also \cite{SwVi}) proved that $e(X)\geq n+1$.
If the distance of $X$ to $\ell^n_p$ is small, then Swanepoel and Villa \cite{SwVi} proved that $e(X)\geq n$. Moreover,
it is known that $e(X)\rightarrow\infty$ as $n\rightarrow\infty$ (cf.~\cite{Br,De}), and this was quantified in \cite{SwVi} by $e(X)>e^{c\sqrt{\log(n)}}$, for some absolute constant $c>0$. Permutation invariant spaces (cf.~\cite{Ko}) or extremal spaces for Bohnenblust inequality (cf.~\cite[Cor.~2.9]{BrGo} and
\cite[Rmk.~4.1]{BrGo2}) also fulfill $e(X)\geq n+1$.

For a given sequence of subspaces $H_1\subset H_2\subset\cdots\subset H_{n-1}$ of dimensions $1,2,\dots,n-1$
we say that a 0-symmetric convex and compact set $B \subset\mathbb R^n$ has the $(H_1,\dots,H_{n-1})$ intersection property if $B\cap(x+H_i)$ is homothetic to $B\cap H_i$, for every
$x\in\mathbb R^n$. $\ell^n_p$ balls, generalized symmetric prisms or doubled cones, are examples of such sets.
After a suitable rotation, we can suppose that $H_i=\langle e_1,\dots,e_i\rangle$, $i=1,\dots,n-1$, where $e_i$ is the \emph{ith-canonical vector}.
If $S\subset\mathbb R^n$ is a simplex, we say that $S$ is \emph{inscribed} in $B$ if $S\subset B$ and the vertices of $S$ belong to
the \emph{boundary} of $B$, $\partial B$.
Furthermore, we say that $B$ has the $(H_1,\dots,H_{n-1})$-2-intersection property if it has the $(H_1,\dots,H_{n-1})$ intersection property
and if $S_{i-1}$ is an $(i-1)$-dimensional equilateral simplex with diameter $\mathrm{D}(S_{i-1})=1$ and is inscribed in $B\cap(te_i+H_{i-1})$, for some $t>0$, then $(B\cap H_i)\cap(2te_i+H_{i-1})=\emptyset$, for every $i=2,\dots,n$ (here we assume that $H_n=\mathbb R^n$).
When $B$ is smooth, this additional condition relates the curvature of $B$ and the diameter of a suitable $(n-1)$-dimensional equilateral simplex.

We also say that $X$ has the $(H_1,\dots,H_{n-1})$(-2-)intersection
property if its unit ball $B$ has it. Our main result in the paper is an affirmative answer to the conjecture on $e(X)\geq n+1$ for such spaces.
\begin{thm}\label{thm:mainResult}
Let $X$ be an n-dimensional normed space with unit ball $B$, and let $H_1\subset\dots\subset H_{n-1}$ be subspaces of dimensions $1,\dots,n-1$.
If $X$ has the $(H_1,\dots,H_{n-1})$-2-intersection property, then $e(X)\geq n+1$.
\end{thm}
The idea behind has some similarities with the construction of Makeev \cite{Mak} when proving $e(X)\geq 5$ in the 4-dimensional case. 
Indeed, Theorem \ref{thm:mainResult} shows norm spaces where Makeev's idea admits an inductive step in the dimension. Moreover,
we give examples in Remark \ref{rmk:NotInductive} of spaces where such induction cannot be done.

\section{Proof of the main result}

In our proof, we actually show that for spaces having the $(H_1,\dots,H_{n-1})$-2-intersection property,
something slightly stronger than Theorem \ref{thm:mainResult} can be proven. For any normed space $X$ endowed with a norm $\|\cdot\|$, the \emph{diameter}
of a set $A\subset X$ is given by
\[
\mathrm{D}(A)=\sup_{x,y\in A}\|x-y\|.
\]
For any $A\subset\mathbb R^n$, we let $\mathrm{int}(A)$ and $\mathrm{conv}(A)$ be the \emph{interior} and \emph{convex hull} of $A$, respectively.
We let $B^n_2$ be the n-dimensional Euclidean unit ball.

\begin{lem}\label{lem:mainResult}
Let $X$ be an n-dimensional normed space with unit ball $B$, and let $H_1\subset\dots\subset H_{n-1}$ be subspaces of dimensions $1,\dots,n-1$.
If $X$ has the $(H_1,\dots,H_{n-1})$-2-intersection property, then there exists an n-dimensional equilateral simplex $S$ inscribed in $B$ with
$\mathrm{D}(S)>1$.
\end{lem}

The additional condition that $\mathrm{D}(S)>1$ turns out to be \emph{crucial} in the proof. In fact, this condition is the \emph{hidden ingredient}
in the proof of Makeev \cite{Mak} of the 4-dimensional case, where he made use of the control of the diameter of a corresponding 3-dimensional simplex.
Let us also observe that $S\subset B$ already implies that $\mathrm{D}(S)\leq 2$.

Lemma \ref{lem:mainResult} also ensures that we find a large equilateral simplex inscribed in $B$. The existence of inscribed simplices is at the kernel
of the problem of generalizing a bit more this construction, and it has been largely studied (cf.~\cite{Ko2,Mak2}). Let us briefly remark how to construct
an n-dimensional equilateral simplex from a very particular (n-1)-dimensional equilateral simplex.
\begin{rmk}\label{rmk:n+1}
Let $X$ be an n-dimensional normed space with smooth and strictly convex unit ball $B$, and let $H=\langle e_1,\dots,e_{n-1}\rangle$. If
there exists an (n-1)-dimensional equilateral simplex $T$ inscribed in $B\cap H$ with $\mathrm{D}(T)>1$, then $e(X)\geq n+1$.

The sketch of the proof is as follows. We let
$\varphi(t)$ be the diameter of the maximum dilatation of $T$ which can be inscribed in $B\cap(te_n+H)$, for $t\geq 0$, and we will denote by $T_t$
one of those simplices. The existence of $T_t$ is well-known in the smooth strictly convex case (cf.~\cite{KrNe}), but not the uniqueness (this is why we choose the maximum). Of course, $\varphi(t)$ is continuous, and
if $B\cap(te_n+H)\neq\emptyset$ for $t\in[0,t_0]$, then the strict convexity of $B$ implies that $\varphi(t_0)=0$. Thus there exists $t^*\in(0,t_0)$
such that $\varphi(t^*)=1$. Hence, the simplex $S=\mathrm{conv}(\{0\}\cup T_{t^*})$ is an n-dimensional equilateral simplex.
\end{rmk}

\begin{proof}[Proof of Lemma \ref{lem:mainResult}]
We prove it by induction in the dimension $n\geq 2$. Let us start observing that after a suitable rotation, we can suppose that
$H_i=\langle e_1,\dots,e_i\rangle$, for $i=1,\dots,n-1$.

We start with the planar case $n=2$. Since $B$ is a planar set, let us suppose after a rescalation that $\pm e_1\in\partial B$, and
let $M_0=[-e_1,e_1]$.
Clearly $\mathrm{D}(M_0)=2$. Let $M_t=B\cap(te_2+H_1)\neq\emptyset$ for $t\in[0,t_0]$. Since $B$
is 0-symmetric, then $\mathrm{D}(M_0)\geq\mathrm{D}(M_t)$ for every $t\in[0,t_0]$. Moreover, since $B$ is convex, then $\varphi(t)=\mathrm{D}(M_t)$
decreases continuously on $t\in[0,t_0]$. Hence, we may have two cases: either $\varphi(t_0)\leq 1$, or $\varphi(t_0)>1$.

In the first case, we know that since $\varphi(t)$ is continuous, then there exists $t^*\in[0,t_0]$ such that $\varphi(t^*)=1$, i.e., if $M_{t^*}=[x_1,x_2]$, that $\mathrm{D}(M_{t^*})=\|x_1-x_2\|=1$. Since $x_1,x_2\in\partial B$, then $\|x_i\|=1$, $i=1,2$, and thus $S=\mathrm{conv}(\{0,x_1,x_2\})$ gives an equilateral triangle $S\subset B$ with $\mathrm{D}(S)=1$. Moreover, since $(2t^*e_2+H_1)\cap B=\emptyset$, then it is clear that the equilateral triangle
$T_{t^*}=\mathrm{conv}(\{x_1,x_2,x_1+x_2\})$ (which is the reflexion of $\mathrm{conv}(\{0,x_1,x_2\})$ w.r.t.~$(x_1+x_2)/2$) has the vertex $x_1+x_2\notin B$, and that $T_{-t^*}=\mathrm{conv}(\{0,-x_1,-x_2\})$ has its vertex $0\in\mathrm{int}(B)$. If we consider in general $T_t$ to be the corresponding
homothetic triangle of $T_{t^*}$ with horizontal edge inscribed in the section $(te_2+H_1)\cap B$, then by continuity there exists $t'\in(-t^*,t^*)$ such that
the last vertex of $T_{t'}$ belongs to $\partial B$, and thus, such that $T_{t'}$ is inscribed in $B$. Moreover, we observe that
now $\mathrm{D}(T_{t'})>1$; otherwise, if $\mathrm{D}(T_{t'})=1$, then we find that two sections of $B$, $B\cap (t^*e_2+H_1)$ and $B\cap (t'e_2+H_1)$
(whose lengths measured w.r.t.~the norm are 1) are exactly the same, which by the 0-symmetry and convexity of $B$  directly implies that $B\cap H_1$ is also a copy of those sections, and thus, a contradiction (since its length w.r.t.~the norm is 2).

In the second case, since $M_{t_0}\subset\partial B$ with $\mathrm{D}(M_{t_0})>1$, then we pick a line segment $[x_1,x_2]\subset \mathrm{relint}(M_{t_0})$ with $\|x_1-x_2\|=1$. Once more, if $S=\mathrm{conv}(\{0,x_1,x_2\})$, since $\|x_i\|=1$ $i=1,2$, $S$ is an equilateral triangle of diameter 1.
Now we consider bigger homothetic copies of $S$, with an edge contained in $B\cap(t_0e_2+H_1)$. By continuity, either at some point the last vertex belongs to $\partial B$ too (and in that case, that would be the desired equilateral triangle) or we would end up with a homothetic copy
of $S$ with an edge being $B\cap(t_0e_1+H_1)$ and the opposing vertex still belonging to $\mathrm{int}(B)$. Then, we would go on considering
bigger homothetic copies of that triangle, with an edge being equal to $B\cap(te_2+H_1)$, for $t\in[0,t_0]$. It is clear that at some point before arriving at $t=0$ we get such a homothetic triangle $S^*$ of $S$, with its last vertex in $\partial B$, hence being an equilateral triangle inscribed in $B$, with $\mathrm{D}(S^*)>1$. This concludes the case of $n=2$.

Now we prove the general case, assuming the induction hypothesis on lower dimensional spaces.
The idea follows essentially the same steps as the proof of the planar case.
Let us consider $B\cap H_{n-1}$,
which is a $0$-symmetric convex and compact set, with the $(H_1,\dots,H_{n-2})$-2-intersection property.
Moreover, $B\cap H_{n-1}$ induces in $H_{n-1}$ the same norm than the one in $X$. Hence,
the induction hypothesis implies the existence of $x_1,\dots,x_{n}\subset B\cap H_{n-1}$ such that
$T_0$ is an equilateral (n-1)-simplex with
$\mathrm{D}(T_0)>1$ inscribed in $B\cap H_{n-1}$.
Let $M_t=B\cap(te_n+H_{n-1})\neq\emptyset$ for $t\in[0,t_0]$.
Moreover, since $M_t$ is a homothety of $M_0$ let $T_t$ be the corresponding homothety of $T_0$ contained in $M_t$, and whose homothetic ratio with $T_0$ is the same than the one between $M_t$ and $M_0$.
Since $B$ is 0-symmetric and convex, then $\varphi(t)=\mathrm{D}(T_t)$ is continuous and non-increasing on $t\in[0,t_0]$.
Then, we either have that $\varphi(t_0)\leq 1$ or $\varphi(t_0)>1$.

In the first case, since $\varphi(t)$ is continuous and $\varphi(0)>1\geq\varphi(t_0)$, then there exists $t^*\in[0,t_0]$ such that $\varphi(t^*)=1$, i.e.,
$\mathrm{D}(T_{t^*})=1$. If $T_{t^*}=\mathrm{conv}(\{x_1,\dots,x_n\})$, observe that $\|x_i-x_j\|=1$ for every $1\leq i<j\leq n$. Moreover,
since $x_i\in\partial B$,  then $\|x_i\|=1$ for $i=1,\dots,n$, thus $S_{t^*}=\mathrm{conv}(\{0,x_1,\dots,x_n\})$ is an equilateral simplex of diameter 1.
We now consider the homothetic copies $S_t$ of $S_{t^*}$ such that they have the facet parallel to $H_{n-1}$ inscribed in $B\cap(te_n+H_{n-1})$, for every
$t\in[-t^*,t^*]$. By hypothesis, the vertex of $S_{-t^*}$ opposing the facet contained in $-t^*e_n+H_{n-1}$ must be outside $B$ (remember that the distance
from this vertex to $H_{n-1}$ is exactly $2t^*$). Thus, by continuity there exists $t'\in(-t^*,t^*)$ such that the last vertex of $S_{t'}$
belongs to $\partial B$ and thus $S_{t'}$ is inscribed in $B$. Moreover, the homothetic ratio of $S_{t'}$ w.r.t.~$S_{t^*}$ is strictly
greater than 1; otherwise, having two equal sections $B\cap (t^*e_n+H_{n-1})$ and $B\cap (t'e_n+H_{n-1})$,
and since $B$ is 0-symmetric and convex, would directly imply that $B\cap H_{n-1}$ is also a copy of the same size, and thus,
a contradiction with the fact that $\mathrm{D}(T_0)>1$ but $\mathrm{D}(T_{t^*})=1$. Thus $\mathrm{D}(S_{t'})>1$, as desired.

In the second case, since $M_{t_0}\subset\partial B$ with $\mathrm{D}(T_{t_0})>1$, then we select a smaller homothety of $T_{t_0}$ inside $B\cap(t_0e_n+H_{n-1})$, with homothety factor $\rho_0<1$ and called $T_{\rho_0}$, such that $\mathrm{D}(T_{\rho_0})=1$.
If $T_{\rho_0}=\mathrm{conv}(\{x_1,\dots,x_n\})$, then $\|x_i-x_j\|=1$ for every $1\leq i<j\leq n$, and since $x_i\in\partial B$, then
$\|x_i\|=1$ for $i=1,\dots,n$, and thus $S^*=\mathrm{conv}(\{0,x_1,\dots,x_n\})$ is an equilateral simplex.
As we did in the planar case, we now consider bigger homothetic copies of $S^*$ until reaching one of them inscribed in $B$. This can be done as before,
in two steps. First, the homothetic copies of $S^*$ will have their facet parallel to $H_{n-1}$ contained in $t_0e_n+H_{n-1}$.
In this first step, either the opposing vertex touches at some point $\partial B$, and hence that is the desired inscribed simplex, or the vertex still belongs to
$\mathrm{int}(B)$. In such case, we would in a second step consider bigger homothetic copies of $S^*$, so that the facet parallel to $H_{n-1}$
keeps being inscribed in $B\cap(te_n+H_{n-1})$, for $t\in[0,t_0]$. Since clearly the last homothetic copy with its facet inscribed in $B\cap H_{n-1}$
would have its last vertex outside $B$, there must exist some $t\in(0,t_0]$ such that the corresponding simplex has its last vertex contained in $\partial B$, and thus it is the desired equilateral simplex inscribed in $B$, with diameter larger than 1.
\end{proof}

\begin{rmk}\label{rmk:NotInductive}
If we replace in Lemma \ref{lem:mainResult} the condition $(H_1,\dots,H_{n-1})$-2-intersection property by
$(H_1,\dots,H_{n-1})$ intersection property, then the proof idea would fail if $n\geq 3$ (but would still be fine if $n=2$).
To see this, we show it when $n=3$ and the construction would be analogous in higher dimensions.
In particular, the space consider below gives an example of an equilateral simplex which \emph{cannot} be
inscribed in the unit ball.

Let $B=\mathrm{conv}((B^2_2\times\{0\})\cup\{(0,0,\pm 1)\})$. If $H_i=\langle e_i\rangle$, $i=1,2$,
then $B$ has the $(H_1,H_2)$ intersection property, \emph{but not} the $(H_1,H_2)$-2-intersection property.
The triangle $T=\mathrm{conv}(\{(\pm\sqrt{3}/2,-1/2,0),(0,1,0)\})$ is equilateral, inscribed in $B$, and has $\mathrm{D}(T)=\sqrt{3}$.
It is quite easy to check that the triangle
\[
T^*=\frac{1}{\sqrt{3}}T+\left(0,0,\frac{\sqrt{3}-1}{\sqrt{3}}\right)
\]
is also inscribed in $B$ with $\mathrm{D}(T^*)=1$. Thus the tetrahedron $S=\mathrm{conv}(\{0\}\cup T^*)$ is equilateral.
However, there exists no translation and homothety of $S$ inscribed in $B$, since the angle
of the edges of $S$ touching $0$ with the vertical line equals $\arctan\left(\frac{1}{\sqrt{3}-1}\right)\approx 0.938$
is bigger than the angle of the generator segment of the double cone $[(1,0,0),(0,0,1)]$ with the vertical line, which is $\pi/4\approx 0.785$.

Furthermore, let us consider $B^*$ smooth and strictly convex as close as we want from $B$, namely, $B^*=B+\varepsilon B^n_2$, with small $\varepsilon>0$.
We can then find a tetrahedron $S^*$ close to $S$, also with a horizontal facet being an equilateral triangle.
It is not difficult to check that a homothetic copy of $S^*$ inscribed in $B^*$ must touch the point 
$(0,0,-1)+\varepsilon(0,0,-1)$ and is arbitrarily small if $\varepsilon$ is arbitrarily small too.
This shows that Makeev's construction (see also Remark \ref{rmk:n+1}) cannot give an \emph{induction step} in the dimension,
since in this case gives a set with $\mathrm{D}(S^*)<1$, and thus we could not guarantee the existence of an (n-1)-simplex
$T$ inscribed in $B^*$ with $\mathrm{D}(T)=1$.
\end{rmk}

\emph{Acknowledgements:} I would like to thank Rafael Villa for his helpful comments and some fruitful discussions.

\end{document}